\documentclass[10pt]{amsart}
\usepackage{amssymb, amstext, amscd, amsmath}
\usepackage[all]{xy} 

\makeatletter
\def\@cite#1#2{{\m@th\upshape\bfseries%
[{#1\if@tempswa{\m@th\upshape\mdseries, #2}\fi}]}}
\makeatother

\theoremstyle{plain}
\newtheorem{thm}{Theorem}[section]

\newtheorem{lem}[thm]{Lemma}

\theoremstyle{definition}

  \newcommand{\A}{{\mathcal{A}}}
  \newcommand{\B}{{\mathcal{B}}}
  \newcommand{\C}{{\mathcal{C}}}
  \newcommand{\D}{{\mathcal{D}}}

  \newcommand{\R}{{\mathcal{R}}}

\renewcommand{\phi}{\varphi}
\newcommand{\upchi}{{\raise.35ex\hbox{\ensuremath{\chi}}}}

\newcommand{\alg}{\operatorname{alg}}

\newcommand{\Prim }{\operatorname{Prim}}

\newcommand{\ca}{\mathrm{C}^*}

\begin{document}
\title[Jacobson spectrum]{Piecewise conjugacy for multivariable dynamics over the Jacobson spectrum of a $\ca$-algebra}

\author[E.G.Katsoulis]{Elias~G.~Katsoulis}
\address{Dept. Math.\\ East Carolina University\\
Greenville, NC 27858\\USA}
\email{katsoulise@ecu.edu}

\begin{abstract}
We show that if $(\A , \alpha)$ and $(\B, \beta)$ are automorphic multivariable $\ca$-dynamical systems with isometrically isomorphic tensor algebras (or semi crossed products), then the systems are piecewise conjugate over their Jacobson spectrum. This answers a question of Kakariadis and the author.
\end{abstract}


\date{}
\maketitle


\section{Introduction}        \label{Intro}

The concept of a non-selfadjoint operator algebra associated with a multivariable dynamical system is new and yet fruitfull. Such algebras appeared for the first time in the work of Power \cite{Po} and Donsig, Katavolos and Manoussos \cite{DKM} but their systematic study only started recently with the Memoirs of Davidson and the author \cite{DK}. In that paper, given a multivariable dynamical system over a locally compact Hausdorff space, we isolated two associated operator algebras, the tensor algebra and the semi crossed product, and we made the case that various properties of the dynamical system are encoded in these algebras. Rather surprisingly, the classification of such algebras in \cite {DK} has had an impact beyond operator algebras, as witnessed in the work of Cornelissen and Marcolli \cite{CM} on class field theory.

Inspired by \cite{CM, DK}, Kakariadis and the author \cite{KK} extended the study of multivariable dynamics beyond commutative $\ca$-algebras. It turns out that the non-commutative context allows for questions that do not materialize in the commutative setting. One such question revolves around the various notions of a spectrum for a $\ca$-algebra.  In \cite{KK} we showed that  if $(\A , \alpha)$ and $(\B, \beta)$ are multivariable dynamical systems with isometrically isomorphic tensor algebras (or semi crossed products), then the systems are piecewise conjugate over the spectrum, as described by Ernest in \cite{Ern}. However the conjugacy over the Jacobson spectrum, i.e., the primitive ideal space with the hull-kernel topology, was left open and it was asked as a question at the end of the paper \cite[Question 3]{KK}. In this note we observe that the presence of a continuous open map between the spectra, combined with the results of \cite{KK} settles this question in the affirmative.

\section{The main result}
We adhere to the notation of \cite{KK}  and use as references for the properties of the various spectra of $\ca$-algebras the book of Dixmier \cite{Dix}, and Ernest's paper \cite{Ern}.

If $\A$ is a $\ca$-algebra, then $\Prim (\A)$ will denote its Jacobson spectrum. Let $\hat{\A}$ be the collection of all (equivalence classes of non-trivial) irreducible representations of $\A$ and
\begin{equation} \label{theta}
\theta : \hat{\A} \rightarrow \Prim (\A); \hat{\A} \ni \pi \mapsto \ker \pi.
\end{equation}
The space $\hat{\A}$ is equipped with the smallest topology that makes $\theta$ continuous. This forces $\theta$ to be an open mapping as well.

In \cite{KK} we worked exclusively with J. Ernest's picture for the spectrum for a $\ca$-algebra. Let $\R(\A)$ be the collection of all railway representations of $\A$; these are representations equivalent to appropriately large ampliations of irreducible representations of $\A$, all acting on the same Hilbert space. The space $\R(\A)$ is equipped with the topology of pointwise-SOT convergence. If $\phi : \R(\A) \rightarrow \hat{\A}$ is the map that associates a railway representation to the equivalence class of its associated irreducible representation, then \cite{Ern} shows that $\phi$ is continuous and open.

\begin{lem} \label{elem}
Let $X, Y, Z$ be topological spaces and $f, g, h$ maps so that the following diagram
\[
\xymatrix{X \ar[r]^f \ar[rd]_g  & Y \ar[d]^h \\
                           &Z}
\]
commutes. Assume that $f$ is continuous and open, $Z$ is equipped with the quotient topology for $g$ and $h$ is a bijective surjection. Then $h$ is a homeomorphism.
\end{lem}

\begin{proof}
Assume that $U \subseteq Z$ is open. Since $g$ is continuous, $g^{-1}(U)= f^{-1}(h^{-1}(U))$ is also open. Since $f$ is open, we obtain
\[
f(g^{-1}(U))= f \left(f^{-1}(h^{-1}(U))\right)=h^{-1}(U)
\]
is open and so $h$ is continuous.

Now let $U \subseteq Y$ be open. Then,
\begin{align*}
g^{-1}(h(U)) &= \left( g^{-1}(h^{-1})^{-1}\right)(U)  \\
             &= (h^{-1}g)^{-1}(U)=f^{-1}(U)
             \end{align*}
             is open. Since $Z$ is equipped with the quotient topology for $g$, $h(U)$ is open, i.e., $h$ is open.
\end{proof}

If we take $X=\R (\A)$, $Y = \hat{\A}$, $Z= \R(\A)/ \sim$, i.e., unitary equivalence classes of railway representations, $f = \phi$, $g$ the quotient map and $h$ the map that assigns an equivalence class of irreducible representations to the equivalence class of the corresponding railway representation, then Lemma~\ref{elem} shows that the map $h$ is a canonical homeomorphism between $\hat{A}$ and $\R(\A)/ \sim$. In the sequel, we will not be distinguishing between these two spectra.

Let $X$ and $Y$ be topological spaces and let $\sigma= (\sigma_1, \sigma_2, \dots, \sigma_n)$ and $\tau=(\tau_1, \tau_2, \dots , \tau_n)$ be multivariable dynamical systems consisting of selfmaps of $X$ and $Y$ respectively. Davidson and Katsoulis \cite[Definition 3.16]{DK} define $(X, \sigma)$ and $(Y, \tau)$ to be \textit{piecewise conjugate} if there exists a homeomorphism $h \colon X\rightarrow Y$ and an open cover $\{ U_g \mid g \in S_n \}$ of $Y$ so that
\[
\tau_i (y) =h \, \sigma_{g(i)} \, h^{-1} (y), \text{ for each $y \in U_g$, $g \in S_n$ and $1\leq i\leq n$.}
\]

An (automorphic) multivariable $\ca$-dynamical system $(\A, \alpha)$ consists of a $\ca$-algebra $\A$ and $*$-automorphisms $\alpha = (\alpha_1,  \alpha_2 , \dots,\alpha )$ acting on $\A$. Any automorphism (resp. multivariable system) $\alpha$ of $\A$ induces a homeomorphism (resp. multivariable dynamical system) $\hat{\alpha}$ on $\hat{\A} $, that maps the equivalence class $[\rho]$ of a railway representation to $[\rho\alpha]$. Similarly, there is a map $\tilde{\alpha}$ on $\Prim (\A)$ that maps $\ker \pi$ to $\ker \pi \alpha$.

Below we answer affirmatively \cite[Question 3]{KK}. We will not be explaining the operator algebras appearing in the Theorem below as we do not require any of their defining properties. Instead we direct the reader to \cite[Section 2]{KK}.

\begin{thm} Let $(\A , \alpha)$ and $(\B, \beta)$ be multivariable dynamical systems consisting of $*$-automorphisms. If the associated operator algebras $\alg (\A , \alpha)$ and $ \alg (\B, \beta)$ are isometrically isomorphic as operator algebras, then the multivariable dynamical systems $ ( \Prim (\A), \tilde{\alpha}) $ and $ ( \Prim (\B), \tilde{\beta}) $ are piecewise conjugate.

\end{thm}

\begin{proof}
We begin by noticing that for any choice of $\ca$-algebras $\C$, $\D$ and a \break $*$-automorphism $\delta : \C \rightarrow \D$ we have a commutative diagram

\begin{equation} \label{CD}
\begin{CD}
 \hat{\C} @> \theta >>  \Prim (\C) \\
 @V\hat{\delta}  VV              @V\tilde{\delta}  VV                   \\
 \hat{\D} @> \theta>>\Prim (\D)
\end{CD}
\end{equation}
where $\theta$ is defined in (\ref{theta}).

In \cite[Theorem 4.9]{KK} we proved that if $\alg (\A , \alpha)$ and $ \alg (\B, \beta)$ are isometrically isomorphic as operator algebras, then the dynamical systems $ ( \hat{\A},\hat{ \alpha}) $ and $ ( \hat{\B}, \hat{\beta}) $ are piecewise conjugate. Furthermore, it follows from the proof of \cite[Theorem 4.9]{KK} that the homeomorphism $h$ implementing the piecewise conjugacy comes from a $*$-automorphism $\gamma: \A \rightarrow \B$, i.e., $h = \hat{\gamma}$.

Let $\{ U_g \mid g \in S_n \}$ be an open cover of $\hat{\B}$ so that
\[
\hat{\beta}_i (y)=\hat{\gamma} \, \hat{\alpha}_{g(i)}  \, \hat{\gamma}^{-1} (y), \text{ for each $y \in U_g$, $g \in S_n$ and $1\leq i\leq n$.}
\]
Since $\theta$ is open, $\{ \theta(U_g) \mid g \in S_n \}$ be an open cover of $\Prim (\B)$. Furthermore, repeated use of (\ref{CD}) for the appropriate choices of $\C$, $\D$ and $\delta$ implies that
\begin{align*}
\tilde{\beta_i}\theta (y) &=\theta \hat{\beta}_i (y)=\theta \hat{\gamma} \,
                          \hat{\alpha}_{g(i)} \, \hat{\gamma}^{-1}(y) \\
                     &=\tilde{\gamma} \theta \,\hat{\alpha}_{g(i)} \hat{\gamma}^{-1}(y)
                     =\dots= \tilde{\gamma} \, \tilde{\alpha}_{g(i)} \, \tilde{\gamma}^{-1}\theta(y).
                     \end{align*}
Hence, $ \tilde{\beta_i} \mid _{\theta(U_g)} = \tilde{\gamma} \, \tilde{\alpha}_{g(i)} \, \tilde{\gamma}^{-1}\mid _{\theta(U_g)}$ and we are done.
\end{proof}


\end{document}